\documentclass{birkjour}
 \newtheorem{thm}{Theorem}[section]

 \theoremstyle{definition}
 \newtheorem{defn}[thm]{Definition}
 \theoremstyle{remark}

 \numberwithin{equation}{section}

\newcommand{\mymod}[3]{#1 \equiv #2 \kern -0.5em \pmod{#3}}
\newcommand{\mynotmod}[3]{#1 \not \equiv #2 \kern -0.6em \pmod{#3}}

\begin{document}

%
%
%
%
%
%
%
%
%

\title[Unrestricted modified third-order Jacobsthal quaternions]
 {Unrestricted modified third-order Jacobsthal quaternions}

\author[G. Morales]{Gamaliel Morales}

\address{Instituto de Matem\'aticas,\\
Pontificia Universidad Cat\'olica de Valpara\'iso,\\
Blanco Viel 596,\\
Valpara\'iso, Chile.}
\email{gamaliel.cerda.m@mail.pucv.cl}




\begin{abstract}
In this study, we introduce a new class of quaternions associated with the well-known modified third-order Jacobsthal numbers. There are many studies about the quaternions with special integer sequences and their generalizations. All of these studies used consecutive elements of the considered sequences. Here, we extend the usual definitions into a wider structure by using arbitrary modified third-order Jacobsthal numbers. Moreover, we present unrestricted complex modified third-order Jacobsthal numbers. In addition, we give some properties of this type of quaternions and complex modified third-order Jacobsthal numbers, including generating function and Binet-like formula.
\end{abstract}

\maketitle
\section{Introduction}
In 1843, the Irish mathematician Sir W.R. Hamilton introduced quaternions which are members of a non-commutative division algebra as a kind of extension of complex numbers to higher spatial dimensions. Let $\mathbb{R}^{r,s}$ be the vector space with quadratic forms and $Cl_{r,s}(\mathbb{R})$ be the Clifford algebra on ${\mathbb{R}}^{r,s}$. This means that a standard basis $\{e_{i}\}$ has $r$ vectors square to $1$ and $s$ vectors square to $-1$. Quaternions can be constructed as the even subalgebra of the $Cl_{0,3}(\mathbb{R})$. For more details, the monograph \cite{Ham} can be investigated. 

A quaternion $q$ is formally defined as follows:
$$q=r_{0}+r_{1}\textbf{i}+r_{2}\textbf{j}+r_{3}\textbf{k},$$
where $r_{j}$ ($j=0,1,2,3$) are any real numbers, and the standard basis $\{1, \textbf{i}, \textbf{j}, \textbf{k}\}$ satisfies
$$\textbf{i}^{2}=\textbf{j}^{2}=\textbf{k}^{2}=-1,\  \textbf{i}\textbf{j}=-\textbf{j}\textbf{i}=\textbf{k},\ \textbf{j}\textbf{k}=-\textbf{k}\textbf{j}=\textbf{i},\  \textbf{k}\textbf{i}=-\textbf{i}\textbf{k}=\textbf{j}.$$
The conjugate of $q$ is $q^{*}=r_{0}-r_{1}\textbf{i}-r_{2}\textbf{j}-r_{3}\textbf{k}$, and its norm is 
\begin{equation}\label{norm}
Nr(q)=qq^{*}=r_{0}^{2}+r_{1}^{2}+r_{2}^{2}+r_{3}^{2}.  
\end{equation}

In \cite{Hor}, Horadam used another perspective to the above-mentioned subject by defining the Fibonacci and Lucas quaternions in the forms
$$QF_{n}=F_{n}+F_{n+1}\textbf{i}+F_{n+2}\textbf{j}+F_{n+3}\textbf{k}$$ and $$QL_{n}=L_{n}+L_{n+1}\textbf{i}+L_{n+2}\textbf{j}+L_{n+3}\textbf{k},$$ respectively, where $F_{n}$ is the famous $n$-th Fibonacci number and $L_{n}$ is the $n$-th Lucas number. Szynal-Liana and Wloch \cite{Szy} defined Jacobsthal and Jacobsthal-Lucas quaternions as
$$JQ_{n}=J_{n}+J_{n+1}\textbf{i}+J_{n+2}\textbf{j}+J_{n+3}\textbf{k}$$ and $$JLQ_{n}=j_{n}+j_{n+1}\textbf{i}+j_{n+2}\textbf{j}+j_{n+3}\textbf{k},$$ respectively, where $J_{n}$ and $j_{n}$ are $n$-th Jacobsthal and Jacobsthal-Lucas numbers, respectively. Dasdemir and Bilgici \cite{Das} investigated Binet's formulas and generating functions for Unrestricted Fibonacci and Lucas quaternions.There are many papers devoted to the quaternions whose coefficients are some special integer sequences, which are not mentioned here.

Based on these studies, we define a new class of quaternions and Gaussian modified third-order Jacobsthal numbers associated with the modified third-order Jacobsthal numbers. In addition, we give a Binet-like formula, generating functions, and certain identities such as Cassini's identity for these quaternions.

Before giving the main results, let us introduce these numbers briefly. Modified third-order Jacobsthal numbers are of the form  
\begin{equation}\label{mod}
K_{n}^{(3)}=2^{n}+\omega_{1}^{n}+\omega_{2}^{n}=2^{n}+M_{n},
\end{equation}
where 
\begin{equation}\label{mod1}
M_{n}=\left\{ 
\begin{array}{ccc}
2 & \textrm{if} & \mymod{n}{0}{3} \\ 
-1 & \textrm{if} & \mymod{n}{1,2}{3}
\end{array}
\right.
\end{equation}
and $\omega_{1}=\omega_{2}^{2}=\frac{-1+i\sqrt{3}}{2}$ ($i^{2}=-1$). 

In fact, this sequence appears when we study the third-order Jacobsthal numbers with indices from an arithmetic progression, for example,
$$J_{a(n+3)+r}^{(3)}=\left(2^{a}+\omega_{1}^{a}+\omega_{2}^{a}\right)J_{a(n+2)+r}^{(3)}-\left(2^{a}(\omega_{1}^{a}+\omega_{2}^{a})+1\right)J_{a(n+1)+r}^{(3)}+2^{a}J_{an+r}^{(3)},$$
for fixed integers $a$, $r$ with $0\leq r<a$. Note that $K_{n}^{(3)}=J_{n}^{(3)}+2J_{n-1}^{(3)}+6J_{n-2}^{(3)}$, where $J_{n}^{(3)}$ is the $n$-th third-order Jacobsthal number (see, e.g., \cite{Cer}).

In number theory, recall that a modified third-order Jacobsthal number of order $n$, is a number of
the form $2^{n}+\omega_{1}^{n}+\omega_{2}^{n}$, where $n$ is a nonnegative number. This identity is called as the Binet formula for modified third-order Jacobsthal sequence and it comes from the fact that the modified third-order Jacobsthal numbers can also be defined recursively by 
\begin{equation}\label{rec}
K_{n+3}^{(3)}=K_{n}^{(3)}+7\cdot2^{n},
\end{equation}
with initial conditions $K_{0}^{(3)}=3$, $K_{1}^{(3)}=1$ and $K_{2}^{(3)}=3$. Since this recurrence, substituting $n$ by $n+1$ and $n+2$, we obtain the new form
\begin{equation}\label{rec1}
K_{n+4}^{(3)}=K_{n+1}^{(3)}+7\cdot2^{n+1}
\end{equation}
and 
\begin{equation}\label{rec2}
K_{n+5}^{(3)}=K_{n+2}^{(3)}+7\cdot2^{n+2}.
\end{equation}
Subtracting (\ref{rec}) and (\ref{rec1}) to (\ref{rec2}), we have that $K_{n+5}^{(3)}-K_{n+4}^{(3)}-K_{n+3}^{(3)}=7\cdot2^{n}+K_{n+2}^{(3)}-K_{n+1}^{(3)}-K_{n}^{(3)}$ and then 
\begin{equation}\label{rec3}
K_{n+5}^{(3)}=K_{n+4}^{(3)}+K_{n+3}^{(3)}+2K_{n+2}^{(3)},
\end{equation}
other form for the recurrence relation of modified third-order Jacobsthal sequence. See \cite{Cer} for more details and properties of this type of numbers.

\section{Main results}
In this section, our definitions, some concepts, and the fundamental results are presented. First of all, we give the following definition.

\begin{defn}\label{d1}
 For an integer $n\geq 0$, the unrestricted modified third-order Jacobsthal quaternion $\mathcal{QK}_{n}^{(a,b,c)}$ is defined by $$\mathcal{QK}_{n}^{(a,b,c)}=K_{n}^{(3)}+K_{n+a}^{(3)}\textbf{i}+K_{n+b}^{(3)}\textbf{j}+K_{n+c}^{(3)}\textbf{k},$$  where $a$, $b$ and $c$ are any integers.
\end{defn}

According to our definition, we have the following special cases:
\begin{itemize}
    \item For $a=b=c=-n$, the usual modified third-order numbers are obtained $\mathcal{QK}_{n}^{(-n,-n,-n)}=K_{n}^{(3)}$.
    \item For $b=c=-n$, the unrestricted complex modified third-order Jacobsthal numbers are obtained $\mathcal{QK}_{n}^{(a,-n,-n)}=K_{n}^{(3)}+K_{n+a}^{(3)}\textbf{i}$.
    \item For $a=1$ and $b=c=-n$, the complex modified third-order Jacobsthal numbers are obtained $\mathcal{QK}_{n}^{(1,-n,-n)}=K_{n}^{(3)}+K_{n+1}^{(3)}\textbf{i}$.
    \item For $a=1$, $b=2$ and $c=3$, the modified third-order Jacobsthal quaternions in the usual form of quaternions are obtained: $$\mathcal{QK}_{n}^{(1,2,3)}=K_{n}^{(3)}+K_{n+1}^{(3)}\textbf{i}+K_{n+2}^{(3)}\textbf{j}+K_{n+3}^{(3)}\textbf{k}.$$
\end{itemize}

In \cite{Hor}, Horadam introduced the complex Fibonacci and Lucas numbers. These were named Gaussian Fibonacci and Lucas numbers by Berzsenyi \cite{Ber} and they have been generally referred to this name since that day. So we prefer to use it for the second and third sequences. It should be noted that all the results obtained for the generalized modified third-order Jacobsthal quaternions can be reduced to these numbers. To reduce the volume of the current paper, we only take the unrestricted modified third-order Jacobsthal quaternions into account.

\begin{thm} 
For any non-negative integer $n$, we have 
 \begin{equation}\label{t1}
  \mathcal{QK}_{n+3}^{(a,b,c)}=\mathcal{QK}_{n}^{(a,b,c)}+7\cdot 2^{n}\Theta
  \end{equation}
  and
  \begin{equation}\label{t2}
\mathcal{QK}_{n+3}^{(a,b,c)}=\mathcal{QK}_{n+2}^{(a,b,c)}+\mathcal{QK}_{n+1}^{(a,b,c)}+2\mathcal{QK}_{n}^{(a,b,c)},
  \end{equation}
where $\Theta=1+2^{a}\textbf{i}+2^{b}\textbf{j}+2^{c}\textbf{k}$.
\end{thm}
\begin{proof}
Considering the Definition \ref{d1} and Eq. (\ref{rec}), we get
\begin{align*}
\mathcal{QK}_{n}^{(a,b,c)}&=K_{n}^{(3)}+K_{n+a}^{(3)}\textbf{i}+K_{n+b}^{(3)}\textbf{j}+K_{n+c}^{(3)}\textbf{k}\\
&\ \ + \left(K_{n+3}^{(3)}-7\cdot 2^{n}\right)+\left(K_{n+a+3}^{(3)}-7\cdot 2^{n+a}\right)\textbf{i}+\left(K_{n+b+3}^{(3)}-7\cdot 2^{n+b}\right)\textbf{j}\\
&\ \ +\left(K_{n+c+3}^{(3)}-7\cdot 2^{n+c}\right)\textbf{k}\\
&\ \ =\mathcal{QK}_{n+3}^{(a,b,c)}-7\cdot 2^{n}\Theta.
\end{align*} 
The last equation gives Eq. (\ref{t1}). Using the Definition \ref{d1} and Eq. (\ref{rec3}), Eq. (\ref{t2}) can be proved similarly.
\end{proof}

A Binet-like formula for the unrestricted modified third-order Jacobsthal quaternions is given in the following theorem.
\begin{thm}[Binet-like formula]\label{bin}
For any nonnegative integer $n$, the $n$-th unrestricted modified third-order Jacobsthal quaternion is
\begin{align*}
\mathcal{QK}_{n}^{(a,b,c)}&=2^{n}\Theta +\mathcal{QM}_{n}^{(a,b,c)}\\
\mathcal{QM}_{n}^{(a,b,c)}&=M_{n}+M_{n+a}\textbf{i}+M_{n+b}\textbf{j}+M_{n+c}\textbf{k}=\omega_{1}^{n}\Phi_{1}+\omega_{2}^{n}\Phi_{2},
\end{align*}
where $\Phi_{1}=1+\omega_{1}^{a}\textbf{i}+\omega_{1}^{b}\textbf{j}+\omega_{1}^{c}\textbf{k}$ and $\Phi_{2}=1+\omega_{2}^{a}\textbf{i}+\omega_{2}^{b}\textbf{j}+\omega_{2}^{c}\textbf{k}$.
\end{thm}
\begin{proof}
Using Eqs. (\ref{mod}), (\ref{mod1}) and Definition \ref{d1}, we can write
\begin{align*}
\mathcal{QK}_{n}^{(a,b,c)}&=K_{n}^{(3)}+K_{n+a}^{(3)}\textbf{i}+K_{n+b}^{(3)}\textbf{j}+K_{n+c}^{(3)}\textbf{k}\\
&=\left(2^{n}+M_{n}\right)+\left(2^{n+a}+M_{n+a}\right)\textbf{i}\\
&\ \ + \left(2^{n+b}+M_{n+b}\right)\textbf{j}+\left(2^{n+c}+M_{n+c}\right)\textbf{k}\\
&=2^{n}\left( 1+2^{a}\textbf{i}+2^{b}\textbf{j}+2^{c}\textbf{k}\right)+\mathcal{QM}_{n}^{(a,b,c)}\\
&=2^{n}\Theta +\mathcal{QM}_{n}^{(a,b,c)}
\end{align*}
and 
\begin{align*}
\mathcal{QM}_{n}^{(a,b,c)}=\omega_{1}^{n}\left(1+\omega_{1}^{a}\textbf{i}+\omega_{1}^{b}\textbf{j}+\omega_{1}^{c}\textbf{k}\right)+\omega_{2}^{n}\left(1+\omega_{2}^{a}\textbf{i}+\omega_{2}^{b}\textbf{j}+\omega_{2}^{c}\textbf{k}\right).
\end{align*}
The last equation gives the theorem.
\end{proof}

The following lemma gives the unrestricted modified third-order Jacobsthal quaternions with negative indices.
\begin{thm}
For any integer $n$, we have
\begin{align*}
\mathcal{QK}_{-n}^{(a,b,c)}&=2^{-n}\Theta +M_{n}\left( 1-X_{a-1}\textbf{i}-X_{b-1}\textbf{j}-X_{c-1}\textbf{k}\right)\\
&\ \ + M_{n-1}\left(X_{a}\textbf{i}+X_{b}\textbf{j}+X_{c}\textbf{k}\right),
\end{align*}
where $\Theta$ is as Theorem \ref{t1} and $X_{n}=\frac{\omega_{1}^{n}-\omega_{2}^{n}}{\omega_{1}-\omega_{2}}$.
\end{thm}
\begin{proof}
From Eqs. (\ref{mod}) and (\ref{mod1}), we obtain
\begin{align*}
\mathcal{QK}_{-n}^{(a,b,c)}&=K_{-n}^{(3)}+K_{-n+a}^{(3)}\textbf{i}+K_{-n+b}^{(3)}\textbf{j}+K_{-n+c}^{(3)}\textbf{k}\\
&=\left(2^{-n}+M_{-n}\right)+\left(2^{-n+a}+M_{-(n-a)}\right)\textbf{i}\\
&\ \ + \left(2^{-n+b}+M_{-(n-b)}\right)\textbf{j}+\left(2^{-n+c}+M_{-(n-c)}\right)\textbf{k}\\
&=2^{-n}\left( 1+2^{a}\textbf{i}+2^{b}\textbf{j}+2^{c}\textbf{k}\right)+\mathcal{QM}_{-n}^{(a,b,c)}\\
&=2^{-n}\Theta +\mathcal{QM}_{-n}^{(a,b,c)}.
\end{align*}
Using the identities $M_{-n}=M_{n}$ and $M_{n+m}=X_{m+1}M_{n}-X_{m}M_{n-1}$, we have
\begin{align*}
\mathcal{QM}_{-n}^{(a,b,c)}&=M_{-n}+M_{-(n-a)}\textbf{i}+M_{-(n-b)}\textbf{j}+M_{-(n-c)}\textbf{k}\\
&=M_{n}+M_{n-a}\textbf{i}+M_{n-b}\textbf{j}+M_{n-c}\textbf{k}\\
&=M_{n}\left( 1-X_{a-1}\textbf{i}-X_{b-1}\textbf{j}-X_{c-1}\textbf{k}\right)+M_{n-1}\left(X_{a}\textbf{i}+X_{b}\textbf{j}+X_{c}\textbf{k}\right),
\end{align*}
where $X_{n}$ is the sequence defined by $X_{n}=\frac{\omega_{1}^{n}-\omega_{2}^{n}}{\omega_{1}-\omega_{2}}$. The result follows from the last equation.
\end{proof}

The following theorem gives the norm of a unrestricted modified third-order Jacobsthal quaternion.
\begin{thm}
For any integer $n$, we have
$$Nr\left(\mathcal{QK}_{n}^{(a,b,c)}\right)=\left\lbrace \begin{array}{cc} 
2^{2n}Nr\left(\Theta\right)+2^{n+1}M_{n}\left(1+2^{a}X_{a+1}+2^{b}X_{b+1}+2^{c}X_{c+1}\right)\\
- 2^{n+1}M_{n-1}\left(1+2^{a}X_{a}+2^{b}X_{b}+2^{c}X_{c}\right)\\
+ M_{2n}+M_{2(n+a)}+M_{2(n+b)}+M_{2(n+c)}+8 \end{array} \right\rbrace,$$
where $X_{n}=\frac{\omega_{1}^{n}-\omega_{2}^{n}}{\omega_{1}-\omega_{2}}$.
\end{thm}
\begin{proof}
From Eqs. (\ref{norm}), (\ref{mod}) and (\ref{mod1}), we have
\begin{align*}
Nr\left(\mathcal{QK}_{n}^{(a,b,c)}\right)&=\left(K_{n}^{(3)}\right)^{2}+\left(K_{n+a}^{(3)}\right)^{2}+\left(K_{n+b}^{(3)}\right)^{2}+\left(K_{n+c}^{(3)}\right)^{2}\\
&=\left(2^{n}+M_{n}\right)^{2}+\left(2^{n+a}+M_{n+a}\right)^{2}\\
&\ \ + \left(2^{n+b}+M_{n+b}\right)^{2}+\left(2^{n+c}+M_{n+c}\right)^{2}\\
&=2^{2n}\left(1+2^{2a}+2^{2b}+2^{2c}\right)\\
&\ \ + 2^{n+1}\left(M_{n}+2^{a}M_{n+a}+2^{b}M_{n+b}+2^{c}M_{n+c}\right)\\
&\ \ + M_{n}^{2}+M_{n+a}^{2}+M_{n+b}^{2}+M_{n+c}^{2}.
\end{align*}
Using the identities $M_{n}^{2}=M_{2n}+2$ and $M_{n+m}=X_{m+1}M_{n}-X_{m}M_{n-1}$, we have
\begin{align*}
Nr\left(\mathcal{QK}_{n}^{(a,b,c)}\right)&=2^{2n}Nr\left(\Theta\right)+2^{n+1}M_{n}\left(1+2^{a}X_{a+1}+2^{b}X_{b+1}+2^{c}X_{c+1}\right)\\
&\ \ - 2^{n+1}M_{n-1}\left(1+2^{a}X_{a}+2^{b}X_{b}+2^{c}X_{c}\right)\\
&\ \ + M_{2n}+M_{2(n+a)}+M_{2(n+b)}+M_{2(n+c)}+8,
\end{align*}
where $X_{n}=\frac{\omega_{1}^{n}-\omega_{2}^{n}}{\omega_{1}-\omega_{2}}$ and this equation gives the theorem.
\end{proof}

Now we give some summation formulas for the unrestricted modified third-order Jacobsthal quaternions.
\begin{thm}
For any nonnegative integer $n$, we have
$$\sum_{j=0}^{n}\mathcal{QK}_{j}^{(a,b,c)}=\frac{1}{3}\left\lbrace \begin{array}{cc} \mathcal{QK}_{n+2}^{(a,b,c)}+2\mathcal{QK}_{n}^{(a,b,c)}\\
-3\Theta +(1-\omega_{2})\Phi_{1}+(1-\omega_{1})\Phi_{2}  \end{array} \right\rbrace.$$
\end{thm}
\begin{proof}
Using Theorem \ref{bin} and applying Eqs. (\ref{mod}) for the unrestricted modified third-order Jacobsthal quaternions, we have
\begin{align*}
\sum_{j=0}^{n}\mathcal{QK}_{j}^{(a,b,c)}&=\Theta \sum_{j=0}^{n}2^{j}+\sum_{j=0}^{n}\mathcal{QM}_{k}^{(a,b,c)}\\
&=\Theta \left(2^{n+1}-1\right)+\sum_{j=0}^{n}\left(\Phi_{1}\omega_{1}^{j}+\Phi_{2}\omega_{2}^{j}\right)\\
&=\Theta \left(2^{n+1}-1\right)+\Phi_{1} \left(\frac{\omega_{1}^{n+1}-1}{\omega_{1}-1}\right)+\Phi_{2} \left(\frac{\omega_{2}^{n+1}-1}{\omega_{2}-1}\right)\\
&=\Theta \left(2^{n+1}-1\right)+\frac{1}{3}\left\lbrace \begin{array}{cc} \Phi_{1}\left(\omega_{1}^{n}-\omega_{1}^{n+1}-\omega_{2}+1\right)\\
+\Phi_{2}\left(\omega_{2}^{n}-\omega_{2}^{n+1}-\omega_{1}+1\right)
 \end{array} \right\rbrace.
\end{align*}
Using $\mathcal{QM}_{n}^{(a,b,c)}=\omega_{1}^{n}\Phi_{1}+\omega_{2}^{n}\Phi_{2}$ and $\omega_{1}\omega_{2}=1$, we have
\begin{align*}
\sum_{j=0}^{n}\mathcal{QK}_{j}^{(a,b,c)}&=\frac{1}{3}\left\lbrace \begin{array}{cc}  \mathcal{QK}_{n+2}^{(a,b,c)}+2\mathcal{QK}_{n}^{(a,b,c)}\\
+ \mathcal{QK}_{0}^{(a,b,c)}-\mathcal{QK}_{2}^{(a,b,c)}
 \end{array} \right\rbrace,
\end{align*}
which is the desired result.
\end{proof}

A generating function for the unrestricted modified third-order Jacobsthal quaternions can be found in the following theorem.
\begin{thm}[Generating function]
A generating function for the unrestricted modified third-order Jacobsthal quaternions is
$$\sum_{j=0}^{\infty}\mathcal{QK}_{j}^{(a,b,c)}x^{j}=\frac{\left\lbrace \begin{array}{cc}  
\Theta (1+x+x^{2})\\
\Phi_{1} (1+(\omega_{1}-1)x+2\omega_{2}x^{2})\\
\Phi_{2} (1+(\omega_{2}-1)x+2\omega_{1}x^{2})
 \end{array} \right\rbrace}{1-x-x^{2}-2x^{3}}.$$
\end{thm}
\begin{proof}
Let us define $g_{\mathcal{K}}(x)=\sum_{j=0}^{\infty}\mathcal{QK}_{j}^{(a,b,c)}x^{j}$. Multiplying this equation by $1$, $-x$, $-x^{2}$ and $-2x^{3}$, respectively, and summing the last equations, we obtain
\begin{align*}
(1-x-x^{2}-2x^{3})g_{\mathcal{K}}(x)&=\mathcal{QK}_{0}^{(a,b,c)}+\left(\mathcal{QK}_{1}^{(a,b,c)}-\mathcal{QK}_{0}^{(a,b,c)}\right)x\\
&\ \ + \left(\mathcal{QK}_{2}^{(a,b,c)}-\mathcal{QK}_{1}^{(a,b,c)}-\mathcal{QK}_{0}^{(a,b,c)}\right)x^{2}\\
&=\mathcal{QK}_{0}^{(a,b,c)}+\left(\mathcal{QK}_{1}^{(a,b,c)}-\mathcal{QK}_{0}^{(a,b,c)}\right)x+2\mathcal{QK}_{-1}^{(a,b,c)}x^{2}.
\end{align*}
Also we have
\begin{align*}
2\mathcal{QK}_{-1}^{(a,b,c)}&=2\cdot2^{-1}\Theta +2\omega_{1}^{-1}\Phi_{1}+2\omega_{2}^{-1}\Phi_{2},\\
&=\Theta +2\omega_{2}\Phi_{1}+2\omega_{1}\Phi_{2},\\
\mathcal{QK}_{0}^{(a,b,c)}&=\Theta +\Phi_{1}+\Phi_{2},\\
\mathcal{QK}_{1}^{(a,b,c)}&=2\Theta +\omega_{1}\Phi_{1}+\omega_{2}\Phi_{2}.
\end{align*}
The theorem is proven using the previous equalities.
\end{proof}

We give Cassini's identity for the unrestricted modified third-order Jacobsthal quaternions. We calculate this identity by using the multiplication rules for quaternions.

\begin{thm}[Cassini like-identity for the unrestricted modified third-order Jacobsthal quaternions]
Let $n$ be integer such that $n\geq 1$. Then,
\begin{align*}
\mathcal{QK}_{n+1}^{(a,b,c)}\cdot \mathcal{QK}_{n-1}^{(a,b,c)}-\left(\mathcal{QK}_{n}^{(a,b,c)}\right)^{2}&=\left\lbrace \begin{array}{cc}  
\mathcal{QM}_{n+1}^{(a,b,c)}\cdot \mathcal{QM}_{n-1}^{(a,b,c)}- \left(\mathcal{QM}_{n}^{(a,b,c)}\right)^{2}\\
+ 2^{n}\Theta \left(2\mathcal{QM}_{n+2}^{(a,b,c)}-\mathcal{QM}_{n}^{(a,b,c)}\right)\\
+ 2^{n-1}\left(\mathcal{QM}_{n+1}^{(a,b,c)}-2\mathcal{QM}_{n}^{(a,b,c)}\right)\Theta 
 \end{array} \right\rbrace ,
\end{align*}
where $\mathcal{QM}_{n}^{(a,b,c)}$ is as in Theorem \ref{bin}.
\end{thm}
\begin{proof}
Applying Theorem \ref{bin}, we have that
\begin{align*}
\mathcal{QK}_{n+1}^{(a,b,c)}\cdot \mathcal{QK}_{n-1}^{(a,b,c)}&-\left(\mathcal{QK}_{n}^{(a,b,c)}\right)^{2}\\
&=\left(2^{n+1}\Theta+\mathcal{QM}_{n+1}^{(a,b,c)}\right)\cdot \left(2^{n-1}\Theta+\mathcal{QM}_{n-1}^{(a,b,c)}\right)\\
&\ \ - \left(2^{n}\Theta+\mathcal{QM}_{n}^{(a,b,c)}\right)^{2}\\
&=\mathcal{QM}_{n+1}^{(a,b,c)}\cdot \mathcal{QM}_{n-1}^{(a,b,c)}- \left(\mathcal{QM}_{n}^{(a,b,c)}\right)^{2}\\
&\ \ + 2^{n}\Theta \left(2\mathcal{QM}_{n-1}^{(a,b,c)}-\mathcal{QM}_{n}^{(a,b,c)}\right)\\
&\ \ + 2^{n-1}\left(\mathcal{QM}_{n+1}^{(a,b,c)}-2\mathcal{QM}_{n}^{(a,b,c)}\right)\Theta .
\end{align*}
Using $\mathcal{QM}_{n-1}^{(a,b,c)}=\mathcal{QM}_{n+2}^{(a,b,c)}$. Finally, we obtain
\begin{align*}
\mathcal{QK}_{n+1}^{(a,b,c)}\cdot \mathcal{QK}_{n-1}^{(a,b,c)}&-\left(\mathcal{QK}_{n}^{(a,b,c)}\right)^{2}\\
&=\mathcal{QM}_{n+1}^{(a,b,c)}\cdot \mathcal{QM}_{n-1}^{(a,b,c)}- \left(\mathcal{QM}_{n}^{(a,b,c)}\right)^{2}\\
&\ \ + 2^{n}\Theta \left(2\mathcal{QM}_{n+2}^{(a,b,c)}-\mathcal{QM}_{n}^{(a,b,c)}\right)\\
&\ \ + 2^{n-1}\left(\mathcal{QM}_{n+1}^{(a,b,c)}-2\mathcal{QM}_{n}^{(a,b,c)}\right)\Theta .
\end{align*}
The proof is proved.
\end{proof}

Using the definition \ref{d1} and properties of sequence $\mathcal{QM}_{n}^{(a,b,c)}$, we can deduce the following relation
\begin{align*}
\mathcal{QM}_{n}^{(a,b,c)}&=\omega_{1}^{n}\Phi_{1}+\omega_{2}^{n}\Phi_{2}\\
&=M_{n}+M_{n+a}\textbf{i}+M_{n+b}\textbf{j}+M_{n+c}\textbf{k}\\
&=M_{n}\left(1+X_{a+1}\textbf{i}+X_{b+1}\textbf{j}+X_{c+1}\textbf{k}\right)\\
&\ \ - M_{n-1}\left(X_{a}\textbf{i}+X_{b}\textbf{j}+X_{c}\textbf{k}\right),
\end{align*}
where $$X_{n}=\left\{ 
\begin{array}{ccc}
0 & \textrm{if}& \mymod{n}{0}{3} \\ 
1& \textrm{if} & \mymod{n}{1}{3} \\ 
-1& \textrm{if} & \mymod{n}{2}{3}
\end{array}
\right. $$

Matrix generators play an important role in the theory of the third-order Jacobsthal numbers and the third-order Jacobsthal quaternions (see, for example \cite{Cer1,Cer2}). We derive the matrix representation of the unrestricted modified third-order Jacobsthal quaternions.

\begin{defn}
Unrestricted modified third-order Jacobsthal quaternion matrix $\textbf{J}_{n}(\mathcal{QK})$ is defined by $$\textbf{J}_{n}(\mathcal{QK})=\left[
\begin{array}{ccc}
\mathcal{QK}_{n+3}^{(a,b,c)}& \mathcal{QK}_{n+4}^{(a,b,c)}-\mathcal{QK}_{n+3}^{(a,b,c)}& 2\mathcal{QK}_{n+2}^{(a,b,c)} \\ 
\mathcal{QK}_{n+2}^{(a,b,c)}& \mathcal{QK}_{n+3}^{(a,b,c)}-\mathcal{QK}_{n+2}^{(a,b,c)}& 2\mathcal{QK}_{n+1}^{(a,b,c)} \\ 
\mathcal{QK}_{n+1}^{(a,b,c)}& \mathcal{QK}_{n+2}^{(a,b,c)}-\mathcal{QK}_{n+1}^{(a,b,c)}& 2\mathcal{QK}_{n}^{(a,b,c)}
\end{array}
\right],$$ for all $n\geq 0$.
\end{defn}

\begin{thm}
Let $n\geq 0$ be an integer. Then
$$\textbf{J}_{n}(\mathcal{QK})=\textbf{J}_{0}(\mathcal{QM})\cdot \left[
\begin{array}{ccc}
1& 1& 2 \\ 
1&0& 0\\ 
0& 1& 0
\end{array}
\right]^{n}.
$$
\end{thm}
\begin{proof}
(by induction on $n$) If $n=0$ then assuming that the matrix to the power 0 is the
identity matrix, the result is obvious. Now assume that for any $r\geq 0$ holds $$\textbf{J}_{r}(\mathcal{QK})=\textbf{J}_{0}(\mathcal{QM})\cdot \left[
\begin{array}{ccc}
1& 1& 2 \\ 
1&0& 0\\ 
0& 1& 0
\end{array}
\right]^{r}.$$ By simple calculation using induction's hypothesis and relation (\ref{t2}), we have 
\begin{align*}
\textbf{J}_{0}(\mathcal{QM})\cdot \left[
\begin{array}{ccc}
1& 1& 2 \\ 
1&0& 0\\ 
0& 1& 0
\end{array}
\right]^{r+1}&=\textbf{J}_{0}(\mathcal{QM})\cdot \left[
\begin{array}{ccc}
1& 1& 2 \\ 
1&0& 0\\ 
0& 1& 0
\end{array}
\right]^{r}\cdot \left[
\begin{array}{ccc}
1& 1& 2 \\ 
1&0& 0\\ 
0& 1& 0
\end{array}
\right]\\
&=\textbf{J}_{r}(\mathcal{QK})\cdot \left[
\begin{array}{ccc}
1& 1& 2 \\ 
1&0& 0\\ 
0& 1& 0
\end{array}
\right]\\
&=\textbf{J}_{r+1}(\mathcal{QK}),
\end{align*}
which ends the proof.
\end{proof}

\section{Appendix}
We mentioned above that modified third-order Jacobsthal quaternions can be defined by taking $a=1$, $b=2$ and $c=3$. In this section, we calculate some properties of modified third-order Jacobsthal quaternions. Let $K_{n}^{(3)}$ be the $n$-th modified third-order Jacobsthal number. We have
\begin{align*}
\mathcal{QK}_{n}^{(1,2,3)}&=K_{n}^{(3)}+K_{n+1}^{(3)}\textbf{i}+K_{n+2}^{(3)}\textbf{j}+K_{n+3}^{(3)}\textbf{k},\\
\mathcal{QK}_{n+3}^{(1,2,3)}&=\mathcal{QK}_{n}^{(1,2,3)}+7\cdot 2^{n}(1+2\textbf{i}+4\textbf{j}+8\textbf{k}),
\end{align*}
\begin{align*}
\mathcal{QK}_{n}^{(1,2,3)}&=2^{n}(1+2\textbf{i}+4\textbf{j}+8\textbf{k}) +\mathcal{QM}_{n}^{(1,2,3)},\\
\mathcal{QM}_{n}^{(1,2,3)}&=\omega_{1}^{n}\Phi_{1}+\omega_{2}^{n}\Phi_{2},
\end{align*}
\begin{align*}
Nr\left(\mathcal{QK}_{n}^{(1,2,3)}\right)&=\left\lbrace \begin{array}{cc} 
85\cdot 2^{2n}+7\cdot 2^{n+1}M_{n}+2^{n+1}M_{n-1}\\
+ M_{2n}+M_{2(n+1)}+M_{2(n+2)}+M_{2(n+3)}+8 \end{array} \right\rbrace,\\
\sum_{j=0}^{n}\mathcal{QK}_{j}^{(1,2,3)}&=\frac{1}{3}\left\lbrace \begin{array}{cc} \mathcal{QK}_{n+2}^{(1,2,3)}+2\mathcal{QK}_{n}^{(1,2,3)}\\
-3\Theta +(1-\omega_{2})\Phi_{1}+(1-\omega_{1})\Phi_{2}  \end{array} \right\rbrace\\
\sum_{j=0}^{\infty}\mathcal{QK}_{j}^{(1,2,3)}x^{j}&=\frac{\left\lbrace \begin{array}{cc}  
\Theta (1+x+x^{2})\\
\Phi_{1} (1+(\omega_{1}-1)x+2\omega_{2}x^{2})\\
\Phi_{2} (1+(\omega_{2}-1)x+2\omega_{1}x^{2})
 \end{array} \right\rbrace}{1-x-x^{2}-2x^{3}},
\end{align*}
where $\Phi_{1}=1+\omega_{1}\textbf{i}+\omega_{1}^{2}\textbf{j}+\textbf{k}$ and $\Phi_{2}=1+\omega_{2}\textbf{i}+\omega_{2}^{2}\textbf{j}+\textbf{k}$.

\section{Conclusion}
We defined new quaternions by using definitions of modified third-order Jacobsthal sequence and modified third-order Jacobsthal numbers. The properties of those quaternions were examined. Some theorems about these numbers were presented. Here, we extend the usual definitions into a wider structure by using arbitrary modified third-order Jacobsthal quaternions.

\section{Declarations}
\textbf{Conflict of interest}. The author declares that he has no conflict of interest.




\begin{thebibliography}{1}
\bibitem{Ber} 
G. Berzsenyi, \textit{Gaussian Fibonacci numbers}, Fibonacci Quart. \textbf{15} (1977), 233--236.
\bibitem{Das} 
A. Dasdemir, G. Bilgici \textit{Unrestricted Fibonacci and Lucas Quaternions}, Fundamental Journal of Mathematics and Applications \textbf{4} (2021), 1--9.
\bibitem{Hor} 
A. F. Horadam, \textit{Complex Fibonacci numbers and Fibonacci quaternions}, Amer. Math. Monthly. \textbf{70} (1963), 289--291.
\bibitem{Ham} 
W. R. Hamilton, \textit{Lectures on Quaternions}, 3rd Edition, Hodges and Smith, Dublin, 1853.
\bibitem{Cer} 
G. Morales, \textit{A note on modified third-order Jacobsthal numbers}, Proyecciones. J. Math. \textbf{39} (2020), 731--747.
\bibitem{Cer1} 
G. Morales, \textit{Identities for Third Order Jacobsthal Quaternions}, Advances in Applied Clifford Algebras \textbf{27} (2017) 1043--1053.
\bibitem{Cer2}
G. Morales, \textit{Third-order Jacobsthal Generalized Quaternions}, J. Geom. Symmetry Phys. \textbf{50} (2018) 11--27.
\bibitem{Szy} 
A. Szynal-Liana, I. Wloch, \textit{A note on Jacobsthal quaternions}, Adv. Appl. Clifford Algebr. \textbf{26} (2016), 441--447.

\end{thebibliography}
\end{document}